\documentclass[letterpaper, 10 pt, conference]{ieeeconf}

\IEEEoverridecommandlockouts

\usepackage{graphicx}

\title{\LARGE \bf
Informativity for Data-driven Prediction
}

\author{Joel Stevens and Jeremy Coulson%
\thanks{Both authors are with the Department of Electrical and Computer Engineering, University of Wisconsin–Madison,
        Madison, WI 53706, USA
        {\tt\small \{wjstevens,jeremy.coulson\}@wisc.edu}}}

\usepackage{amsmath, empheq}
\usepackage{amssymb}

\usepackage{amsthm}
\usepackage{mathtools}
\usepackage{mathrsfs}
\usepackage[hidelinks]{hyperref}
\usepackage{enumerate}
\usepackage{algpseudocode}
\usepackage{algorithm}
\usepackage{tikz}
\usepackage{pgfplots}
\pgfplotsset{compat=1.18}
\usepgfplotslibrary{groupplots}
\definecolor{matblue}{RGB}{31,119,180}
\definecolor{matorange}{RGB}{255,127,14}
\definecolor{ybadcolor}{RGB}{175,0,0}

\usepackage{caption}
\hypersetup{
	colorlinks = true,
	linkcolor = black,
	urlcolor = cyan
}

\newcommand{\genHank}[3]{\mathcal{H}_{#1}\left(#2_{#3}\right)}
\newcommand{\LTIq}{\mathsf{LTI}^{m, p}}
\newcommand{\rank}{\textrm{rank}\,}

\newcommand{\im}{\textrm{im}\,}

\newcommand{\Rn}{\mathbb{R}^n}
\newcommand{\Rm}{\mathbb{R}^m}
\newcommand{\Rp}{\mathbb{R}^p}
\newcommand{\Rq}{\mathbb{R}^q}
\newcommand{\B}{\mathcal{B}}
\newcommand{\R}{\mathbb{R}}
\newcommand{\M}{\mathcal{M}}

\newcommand{\setBuild}[2]{\left\{#1 \: \middle| \: #2 \right\}}

\newcommand{\wholes}{\mathbb{Z}_{\geq 0}}

\newcommand{\Input}{\State \textbf{Input: }}
\newcommand{\suchthat}{\textrm{ s.t. }}
\newcommand{\assumptions}{Let $(u^\data, y^\data)\in\B^*_{[0, T - 1]}$, $(u^\ini, y^\ini)\in\B^*_{[0, \Tini - 1]}$, and $u^\f\in\R^{m\Tf}$. }

\newcommand{\Tini}{T_{\textrm{ini}}}
\newcommand{\Tf}{T_{\textrm{f}}}
\newcommand{\ini}{\textrm{ini}}
\newcommand{\f}{\textrm{f}}
\newcommand{\data}{\textrm{d}}
\newcommand{\Lmax}{L_{\max}}
\newcommand{\Up}{U_\textrm{p}}
\newcommand{\Uf}{U_\textrm{f}}
\newcommand{\Yp}{Y_\textrm{p}}
\newcommand{\Yf}{Y_\textrm{f}}

\newcommand{\inform}{$(u^\data, y^\data)$ is informative for unique prediction of $u^\f$ from $(u^\ini, y^\ini)$}

\theoremstyle{definition}
\newtheorem{theorem}{Theorem}

\newtheorem{ex}{Example}

\newtheorem{definition}{Definition}

\newtheorem{lemma}{Lemma}

\newtheorem{prob}{Problem}

\AtEndEnvironment{ex}{\hfill$\bullet$}

\usepackage[acronym]{glossaries}
\newacronym{LTI}{LTI}{linear time-invariant}
\newacronym{DeePC}{DeePC}{Data-enabled Predictive Control}
\newacronym{SPC}{SPC}{Subspace Predictive Control}

\begin{document}
\maketitle
\thispagestyle{empty}
\pagestyle{empty}

\begin{abstract}
    In this work we examine the problem of data-driven prediction. That is, given a \gls{LTI} system with unknown dynamics, we wish to use data collected from the system to predict the system's output response to a given sequence of known inputs. Current methods for predicting require strong conditions on the data such as persistency of excitation. We examine this problem with the goal of finding weaker conditions that still enable prediction. We approach the problem from the data informativity perspective and formally define the notion of informativity for unique prediction. We provide sufficient conditions for informativity for unique prediction and design algorithms to compute the unique output trajectory of the unknown system given known inputs. We demonstrate the results with a numerical example showing that unique output prediction is possible without being able to uniquely identify the unknown data-generating system.
\end{abstract}

\section{Introduction}
Predicting the behavior of dynamical systems from data is a fundamental problem spanning control~\cite{hou2013datadriven}, system identification~\cite{ljung1998system}, machine learning~\cite{vapnik2013nature}, and signal processing~\cite{markovsky2021behavioral}. The prediction problem is to determine future output trajectories of an unknown dynamical system given known future inputs. In control applications, prediction plays a central role in decision-making, enabling optimization over future trajectories and lies at the core of modern data-driven control.

Approaches such as data-driven simulation~\cite{data_driven_simulation}, \gls{DeePC}~\cite{DeePC}, and \gls{SPC}~\cite{favoreel1999spc} demonstrate that accurate predictions can be obtained directly from measured trajectories without explicitly identifying a model of the underlying system. These methods build on the behavioral representation of dynamical systems~\cite{poldermanBook} and use collected data to construct predictors. However, their guarantees rely on the data being \emph{persistently exciting}~\cite{willems2005note}, meaning that the dataset is sufficiently rich to uniquely identify the system dynamics.
While persistent excitation provides a powerful and general foundation for data-driven prediction, obtaining such data may be impractical in many applications. Generating sufficiently rich excitation signals may be costly, time-consuming, or incompatible with safety and operational constraints. When the available data are not persistently exciting, existing prediction methods such as data-driven simulation, \gls{DeePC}, and \gls{SPC} generally lose their theoretical guarantees and may fail to produce reliable predictions.

Importantly, the absence of persistently exciting data does not necessarily prevent solving control problems. In many situations, useful system properties can still be inferred from limited data even when the underlying dynamics cannot be uniquely identified. This raises a fundamental question:
\emph{When does a dataset uniquely determine the future output response of an unknown system to future inputs, even when the system itself cannot be uniquely identified?}
The informativity framework~\cite{informativity} provides a natural tool to study such questions by characterizing what properties of a dynamical system can be inferred from data. Rather than identifying a single model, this framework considers the set of all systems that are consistent with the observed data.

In this paper, we study the data-driven prediction problem from the informativity perspective. Specifically, we introduce the notion of informativity for unique prediction, and investigate conditions under which the available data allow for \emph{unique prediction} of future output trajectories. Our approach characterizes the set of \emph{explaining systems}, i.e., systems consistent with the collected data, and derives conditions under which all explaining systems produce identical output responses for a given future input sequence.

The contributions of this paper are as follows. First, we provide a necessary condition for explaining systems to yield unique output predictions (Lemma~\ref{lemma:lagAndTini}). Second, we characterize the set of trajectories shared by all explaining systems (Lemma~\ref{lemma:image_of_data}). Third, we derive sufficient conditions on the data under which unique prediction is possible (Lemma~\ref{lemma:ini_and_data_in_use} and Theorem~\ref{main_theorem}) and provide algorithms for computing the corresponding predictions (Algorithm~\ref{alg:predict} and Algorithm~\ref{alg:weaving}).

In Section~\ref{sec:prelim} we present preliminaries on \gls{LTI} behaviors. Section~\ref{sec:problem_statement} formalizes the data-driven prediction problem. In Section~\ref{sec:informativity_for_prediction} we provide sufficient conditions on the data under which the data are informative for unique prediction and develop algorithms which can compute the unique prediction. Section~\ref{sec:numerical_example} provides a numerical example that demonstrates the algorithms provided in the previous section. In Section~\ref{sec:future_work} we conclude, outlining directions for future work.

\noindent\textit{Notation:} Denote the set of matrices with $m$ rows and $n$ columns as $\mathbb{R}^{m\times n}$ and the $m\times n$ zero matrix with all zero entries is denoted by $0_{m\times n}$. We write $M\in\R^{m\times \bullet}$ when the number of columns is clear from context. The set $\wholes$ refers to the set of non-negative integers and the set $(\mathbb{R}^q)^{\wholes}$ is the set of all mappings from $\wholes$ to $\mathbb{R}^q$. Void matrices are matrices with 0 rows or 0 columns. For $r,s,t\in\wholes$ we denote the $0 \times r$ and $s \times 0$ void matrices as $0_{0 \times r}$ and $0_{s \times 0}$, respectively. Matrix multiplication extends to void matrices by following these rules for all $P\in\mathbb{R}^{r \times t}, Q\in\mathbb{R}^{t \times s}$: (i) $0_{0 \times r}P = 0_{0 \times t}$, (ii) $Q0_{s \times 0} = 0_{t \times 0}$, (iii) and $0_{s \times 0}0_{0 \times r} = 0_{s \times r}$.
The image of any void matrix is defined to be the set containing only the $0$-vector of appropriate size. Additionally, the rank of any void matrix is defined as 0.

\section{Preliminaries}\label{sec:prelim}
Consider the \gls{LTI} system
\begin{equation}\label{eq:state_space}
    \begin{aligned}
        x(t + 1) &= Ax(t) + Bu(t),\\
        y(t) &= Cx(t) + Du(t),
    \end{aligned}
\end{equation}
where $A \in \mathbb{R}^{n \times n}, B \in \mathbb{R}^{n\times m}, C \in \mathbb{R}^{p\times n}, D \in \mathbb{R}^{p\times m}$, $x(t)\in\Rn$ is the state, $u(t)\in\Rm$ the input, and $y(t)\in\Rp$ the output at time $t\in\wholes$. The \emph{input-output behavior} of~\eqref{eq:state_space} is defined as
\begin{equation}\label{eq:behavior}
    \begin{split}
    \B = \{(u, y)\in(\Rm)^{\wholes}\times (\Rp)^{\wholes}\mid\exists\, x\in(\Rn)^{\wholes}\\\suchthat\eqref{eq:state_space}\text{ holds }\forall t\in\wholes\}.
    \end{split}
\end{equation}
We call~\eqref{eq:state_space} a \emph{state-space representation} of $\B$. Every behavior $\B$ admits multiple state-space representations. We define $n(\B)$ as the minimal state dimension across all state-space representations of $\B$ and call such a representation a \emph{minimal representation} of $\B$. Denote the set of all input-output behaviors of \gls{LTI} systems with $m$ inputs and $p$ outputs by $\LTIq$.

For $t_0, t_1 \in \wholes$ with $t_1 \geq t_0$, define $[t_0, t_1] := \setBuild{t \in \wholes}{t_0 \leq t \leq t_1}$. For $w\in(\Rq)^{\wholes}$, denote $w_{[t_0, t_1]} := (w(t_0), \dots, w(t_1))\in \mathbb{R}^{q\left(t_1 - t_0 + 1\right)}$
as the restriction of $w$ to the interval $[t_0, t_1]$. The \emph{restricted behavior} of an \gls{LTI} system on the interval $[t_0, t_1]$ is defined as $\mathcal{B}_{[t_0, t_1]} :=\setBuild{(u_{[t_0, t_1]}, y_{[t_0, t_1]})}{(u, y)\in\mathcal{B}}$.
We call elements of $\B_{[t_0, t_1]}$ \emph{input-output trajectories} or simply \emph{trajectories}. The restricted behavior $\B_{[t_0, t_1]}$ is a linear subspace of dimension $m(t_1-t_0+1)+n(\B)$~\cite{markovsky2021behavioral}.

We define the \emph{observability matrix} of depth $k$, denoted $\mathcal{O}_k$, as a $0\times n$ void matrix when $k = 0$ and
\[
    \mathcal{O}_ k=
    \begin{pmatrix}
        C\\
        CA\\
        \vdots\\
        CA^{k - 1}
    \end{pmatrix}
\]
for $k > 0$. The lag of an \gls{LTI} behavior $\B$ with minimal state-space representation~\eqref{eq:state_space} is defined as
\begin{equation}\label{lag}
    \ell(\mathcal{B})=\min\setBuild{k}{\rank\mathcal{O}_k = \rank\mathcal{O}_{k + 1}}.
\end{equation}
We present a result that \emph{weaves} two trajectories of a system if they coincide on a long enough interval.
\begin{lemma}[{\cite[Lemma 3]{weaving}}]\label{lemma:weaving}
    Let $\B\in\LTIq$, $T_1,T_2\in\wholes$, $\ell\in\wholes$. Consider $(u^1,y^1)\in\B_{[0,T_1-1]}$, $(u^2,y^2)\in\B_{[0,T_2-1]}$ such that
    \[
    u^1_{[T_1-\ell,T_1-1]}=u^2_{[0,\ell-1]},\text{ and }
    y^1_{[T_1-\ell,T_1-1]}=y^2_{[0,\ell-1]}.
    \]
    If $\ell\geq\ell(\B)$ then
    \[
        u = (u^1, u^2_{[\ell, T_2 - 1]}), \text{ and } y = (y^1, y^2_{[\ell, T_2 - 1]})
    \]
    are such that $(u,y)\in\B_{[0,T_1+T_2-\ell-1]}$.
\end{lemma}

Let $w\in(\Rq)^{\wholes}$, $T\in\wholes$, we define the \emph{depth $k$ Hankel matrix} associated with $w_{[0,T-1]}$ as $\genHank{k}{w}{[0, T - 1]} = 0_{qk\times0}$ if $T < k$ and
\begin{equation*}
    \genHank{k}{w}{[0, T - 1]} = 
    \begin{pmatrix}
        w(0) & w(1) &\cdots & w(T - k)\\
        w(1) & w(2) &\cdots & w(T - k + 1)\\
        \vdots & \vdots & \ddots & \vdots\\
        w(k - 1) & w(k) & \cdots & w(T - 1)
    \end{pmatrix}
\end{equation*}
for $T \geq k$. When building Hankel matrices from data trajectories, we commonly refer to them as \emph{data matrices}.
We define the model class of \gls{LTI} systems with lag at most $L\in\wholes$ and order at most $N\in\wholes$ by
\begin{align*}
    \mathcal{M} = \setBuild{\mathcal{B}\in\LTIq}{\ell(\mathcal{B})\leq L, n(\mathcal{B})\leq N}.
\end{align*}
Given an input-output trajectory of length $T\in\wholes$, $(\bar{u}, \bar{y})\in\R^{(m+p)T}$, we denote the set of systems within the model class $\M$ that could have generated this trajectory as
\begin{align*}
    \Sigma(\bar{u}, \bar{y}) = \setBuild{\mathcal{B}\in\mathcal{M}}{(\bar{u}, \bar{y})\in\mathcal{B}_{[0, T - 1]}}.
\end{align*}
We say the input-output trajectory $(\bar{u},\bar{y})$ is \emph{informative for system identification} if $\Sigma(\bar{u}, \bar{y})$ is a singleton, i.e., there is only one behavior that could have generated $(\bar{u}, \bar{y})$~\cite{camlibel_beyond_2024}.

\section{Problem Statement}\label{sec:problem_statement}
Consider an \gls{LTI} system with unknown behavior $\mathcal{B}^*\in\LTIq$. This system will be referred to as the \emph{true system}. Given
\begin{itemize}
    \item a data trajectory $(u^\data, y^\data)\in\mathcal{B}^*_{[0, T - 1]}$,
    \item an initial trajectory $(u^\ini, y^\ini)\in\mathcal{B}^*_{[0, \Tini - 1]}$ of length $\Tini$,
    \item a sequence of $\Tf$ future inputs $u^\f\in\mathbb{R}^{m\Tf}$,
    \item an upper bound $L\in\wholes$ on the lag $\ell(\B^*)$, and
    \item an upper bound $N\in\wholes$ on the order $n(\B^*)$,
\end{itemize}
we seek to find a unique future response $y^\f\in\mathbb{R}^{p\Tf}$ such that $(u^\ini, u^\f, y^\ini, y^\f)\in\mathcal{B}^*_{[0, \Tini + \Tf - 1]}$. 
If the data is informative for system identification, then unique output response can be computed using methods from~\cite{data_driven_simulation}. However, informativity for sytem identification is not always necessary to uniquely predict future output responses as outlined in the following example.
For ease of notation, we write $\Sigma_\data = \Sigma(u^\data, y^\data)$ and $\Sigma_\ini = \Sigma(u^\ini, y^\ini)$.
\begin{ex}\label{ex:predict-without-ID}
    Consider the \gls{LTI} system $\B^*$ with state-space representation
    \begin{equation}\label{eq:sys1}
        \begin{aligned}
            x(t + 1) &= x(t) + u(t),\\
            y(t) &= x(t),
        \end{aligned}
    \end{equation}
    where $n=m=p=1$. Take $L=1$ and $N=1$. Clearly $\ell(\B^*)\leq L$ and $n(\B^*)\leq N$. Let $T=3$ and consider the data trajectory $(u^\data,y^\data)\in\B^*_{[0,T-1]}$ with 
    \begin{equation}\label{eq:data}
    \begin{aligned}
    (u^\data(0),u^\data(1),u^\data(2))&= (1, -1, 1), \\
    (y^\data(0),y^\data(1),y^\data(2))&= (0, 1, 0).
    \end{aligned}
    \end{equation}
    Let $\Tini=1$. Consider the initial trajectory $(u^\ini,y^\ini)\in\B^*_{[0,\Tini-1]}$ with $u^\ini = -2$ and $y^\ini = 1$. Let $\Tf=2$. Consider the future inputs $u^\f(0) = 2, u^\f(1)= -2$. We can explicitly compute the sets of explaining systems as
    \begin{equation*}
    \Sigma_\data\cap\Sigma_\ini=\left\{ \B \;\middle|\;
{\small\begin{array}{@{}l@{}}
   \exists r\in\R\text{ such that } \forall(u,y)\in\B,\;\forall t\in\wholes,\\
   y(t + 1) - y(t) = ru(t + 1) + (r + 1)u(t)
\end{array}}
\right\}.
    \end{equation*}
    Clearly, $(u^\data,y^\data)$ is not informative for system identification as there are many other systems consistent with the data. Therefore, on the basis of $(u^\data, y^\data)$ and $(u^\ini, y^\ini)$, we cannot differentiate between any of the explaining systems. Thus, we must find a unique output response $y^\f$ such that $(u^\ini, u^\f, y^\ini, y^\f)\in\mathcal{B}_{[0, \Tini + \Tf - 1]}$ for all $\B\in\Sigma_\data\cap\Sigma_\ini$. Indeed, $y^\f(0) = -1$, $y^\f(1)=1$ is an output response that is consistent with all explaining systems. We can prove that indeed this $y^\f$ is the unique output for all explaining systems and it can be obtained directly from the data, without explicitly computing the set of explaining systems $\Sigma_\data\cap\Sigma_\ini$ (see Theorem~\ref{main_theorem}).
\end{ex}

Motivated by Example~\ref{ex:predict-without-ID} and inspired by the informativity framework~\cite{informativity}, we define a notion of informativty for unique prediction. 
\begin{definition}\label{def_informativity}
    Let $T,\Tini, \Tf\in\wholes$, $(u^\data, y^\data)\in\mathcal{B}^*_{[0, T - 1]}$, and $(u^\ini, y^\ini)\in\mathcal{B}^*_{[0, \Tini - 1]}$, $u^\f\in\mathbb{R}^{m\Tf}$. We say the data $(u^\data, y^\data)$ is \emph{informative for unique prediction of} $u^
    \f$ \emph{from} $(u^\ini, y^\ini)$ if
    \begin{enumerate}[(i)]
        \item\label{def:i} for every $\mathcal{B}\in\Sigma_\data\cap\Sigma_\ini$ there exists a unique $y^\f\in\mathbb{R}^{p\Tf}$ such that $(u^\ini, u^\f, y^\ini, y^\f)\in\mathcal{B}_{[0, \Tini + \Tf - 1]}$, and
        \item\label{def:ii} there exists a $y^\f\in\mathbb{R}^{p\Tf}$ such that $(u^\ini, u^\f, y^\ini, y^\f)\in\mathcal{B}_{[0, \Tini + \Tf - 1]}$ for all $\mathcal{B}\in\Sigma_\data\cap\Sigma_\ini$.
    \end{enumerate}
\end{definition}

Condition (\ref{def:i}) asks that every explaining system have their own unique response to $u^\f$. On the other hand, condition (\ref{def:ii}) asks that every explaining system have a response that they all agree on. We highlight the importance of both conditions through an example.
\begin{ex}
    Consider the \gls{LTI} system $\B^*$ with state-space representation
    \begin{align*}
        x(t + 1) &=
        \begin{pmatrix}
            0 & 1\\
            1 & 0
        \end{pmatrix}
        x(t) + 
        \begin{pmatrix}
            1\\
            1
        \end{pmatrix}
        u(t),\\
        y(t) &=
        \begin{pmatrix}
            0 & 1\\
        \end{pmatrix}
        x(t),
    \end{align*}
    where $m=p=1$ and $n=2$. Take $L = 2$ and $N = 2$. Clearly $\ell(\B^*)\leq L$ and $n(\B^*)\leq N$. Let $T=2$ and consider the data trajectory $(u^\data,y^\data)\in\B^*_{[0,T-1]}$ with $u^\data(0)=u^\data(1) =0$ and $y^\data(0)=y^\data(1) = 0$. Let $\Tini=1$. Consider the initial trajectory $(u^\ini,y^\ini)\in\B^*_{[0,\Tini-1]}$ with $u^\ini(0) = 0$ and $y^\ini(0) = 0$. Let $\Tf=1$. Consider the future input $u^\f=0$. Note that $(u^\data,y^\data)\in\B_{[0,1]}$ and $(u^\ini,y^\ini)\in\B_{[0,0]}$ for all $\mathcal{B}\in\mathcal{M}$ by linearity. Hence, $\Sigma_\data = \Sigma_\ini = \mathcal{M}$. We see that (\ref{def:i}) does not hold for this example by finding two different outputs $y^\f$ that the true system $\B^*$ could have as responses to $u^\f$ depending on the initial state. If the initial state is $x(0) = (0, 0)$ then $y^\f = 0$ is the output of $\B^*$, but if $x(0) = (1, 0)$ then $y^\f = 1$ is the output of $\B^*$. On the basis of the initial trajectory $(u^\ini,y^\ini)$, it is not possible to distinguish between these two cases. Despite this, condition (\ref{def:ii}) holds, as the system $\B = (\Rm)^{\wholes}\times\{0\}^{\wholes}\in\Sigma_\data\cap\Sigma_\ini=\M$ is an explaining system. Furthermore, the only possible response of this $\B$ to $u^\f$ is $y^\f = 0$, and since every \gls{LTI} system contains the zero trajectory by linearity, $y^\f = 0$ is the only $y^\f$ satisfying condition (\ref{def:ii}). Therefore, only satisfying (\ref{def:ii}) is not sufficient for uniquely predicting $y^\f$. On the other hand if (\ref{def:i}) holds, but (\ref{def:ii}) fails then there is no trajectory that all explaining systems agree on and therefore we cannot predict which will be the future response $y^\f$ of the true system $\B^*$
\end{ex}
We now pose the problem addressed in this paper.

\begin{prob}\label{problem}
    Find sufficient conditions on the data $(u^\data,y^\data)$, initial trajectory $(u^\ini, y^\ini)$, future inputs $u^\f$, upper bounds $L$, and $N$ such that \inform.
\end{prob}

\section{Informativity for Prediction}\label{sec:informativity_for_prediction}
We first present a necessary condition for the uniqueness of the predicted output required in Definition~\ref{def_informativity} condition (\ref{def:i}).
\begin{lemma}\label{lemma:lagAndTini}
    Let $\B\in\LTIq$ with state space realization~\eqref{eq:state_space}. Let $\Tini, \Tf \in \wholes$. Suppose $(u^\ini, y^\ini) \in \mathcal{B}_{[0, \Tini - 1]}, u^\f \in \mathbb{R}^{m\Tf}$. Then there exists a unique $y^\f$ such that $(u^\ini, u^\f, y^\ini, y^\f) \in \mathcal{B}_{[0, \Tini + \Tf - 1]}$ if and only if $\Tini \geq \ell(\mathcal{B})$.
\end{lemma}
\begin{proof}
    $(\Leftarrow):$ This direction is proved in~\cite[Proposition 6]{data_driven_simulation}. $(\Rightarrow):$ Assume there exists a unique $y^\f$ such that $(u^\ini, u^\f, y^\ini, y^\f) \in \mathcal{B}_{[0, \Tini + \Tf - 1]}$. Assume towards a contradiction that $\Tini < \ell(\mathcal{B})$. By definition of lag~\eqref{lag}, we have that $\dim\ker \mathcal{O}_{\Tini} > \dim\ker\mathcal{O}_{\Tini + 1}$. Since $(u^\ini, y^\ini) \in \mathcal{B}_{[0, \Tini - 1]}$, there exist initial states $x^\ini, z^\ini \in \Rn$ and $v \in \ker\mathcal{O}_{\Tini}\setminus\ker\mathcal{O}_{\Tini + 1}$ such that $z^\ini = x^\ini + v$ and
    \begin{align*}
        y^\ini &= \mathcal{O}_{\Tini}x^\ini + \mathcal{T}_{\Tini}u^\ini,\\
        y^\ini &= \mathcal{O}_{\Tini}z^\ini + \mathcal{T}_{\Tini}u^\ini.
    \end{align*}
    We propagate the system forward for both initial states using $u^\f$ to obtain $y^\f$. By uniqueness of $y^\f$,
    \begin{align*}
       y^\f=\mathcal{O}_{\Tf}x^\f + \mathcal{T}_{\Tf}u^\f=\mathcal{O}_{\Tf}z^\f + \mathcal{T}_{\Tf}u^\f
    \end{align*}
    where
    \[
    \begin{aligned}
    x^\f &= A^{\Tini}x^\ini+\begin{pmatrix}
            A^{\Tini - 1}B & \cdots & AB & B
        \end{pmatrix}
        u^\ini, \\
        z^\f &= A^{\Tini}z^\ini +
        \begin{pmatrix}
            A^{\Tini - 1}B &  \cdots & AB & B
        \end{pmatrix}
        u^\ini.
 \end{aligned}
 \]
   Hence, $\mathcal{O}_{\Tf}A^{\Tini}v=0$. However, this contradicts the fact that $CA^{\Tini} v$ is nonzero by definition of $v$, proving the result. 
   \end{proof}

As a consequence of Lemma~\ref{lemma:lagAndTini}, informativity for unique prediction of any $u^\f$ from $(u^\ini,y^\ini)$ requires the initial data length $\Tini$ to be greater than or equal to the lag of every explaining system in $\Sigma_\data\cap\Sigma_\ini$. In particular, $\Tini$ must be larger than or equal to the maximum lag of any explaining system, which we note by  $\Lmax := \setBuild{\ell(\mathcal{B})}{\mathcal{B}\in\Sigma_\data\cap\Sigma_\ini}$. We formalize this in the following lemma.

\begin{lemma}\label{lemma:MinimumTini}
    \assumptions Suppose \inform. Then, $\Tini \geq \Lmax$.
\end{lemma}
\begin{proof}
    Let $y^\f$ be the unique output such that $(u^\ini, u^\f, y^\ini, y^\f) \in \mathcal{B}^*_{[0, \Tini + \Tf - 1]}$. Assume by contradiction that $\Tini < \Lmax $. Then, there exists a system $\B \in \Sigma_\data\cap\Sigma_\ini$ with lag $\ell(\mathcal{B}) > \Tini$. By Lemma \ref{lemma:lagAndTini}, there must exist some $\tilde{y}^\f\in\mathbb{R}^{p\Tf}$ such that $\tilde{y}^\f\neq y^\f$ so that $(u^\ini, u^\f, y^\ini, \tilde{y}^\f)\in\mathcal{B}_{[0, \Tini + \Tf - 1]}$, but this directly contradicts condition (\ref{def:i}) in Definition~\ref{def_informativity}. Therefore, $\Tini \geq \Lmax$. 
\end{proof}
This lemma allows us to state an equivalent set of conditions for informativity for unique prediction.
\begin{lemma}\label{lem:alt_definition}
    \assumptions Then \inform \ if and only if
    \begin{enumerate}[(i)]
        \item\label{cor:i} $\Tini\geq\Lmax$, and
        \item\label{cor:ii} there exists a $y^\f\in\mathbb{R}^{p\Tf}$ such that $(u^\ini, u^\f, y^\ini, y^\f)\in\mathcal{B}_{[0, \Tini + \Tf - 1]}$ for all explaining systems $\mathcal{B}\in\Sigma_\data\cap\Sigma_\ini$.
    \end{enumerate}
\end{lemma}
\begin{proof}
    ($\Rightarrow$): This direction follows from Lemma \ref{lemma:MinimumTini}. ($\Leftarrow$): Since condition (\ref{cor:ii}) is identical to Definition \ref{def_informativity} condition (\ref{def:ii}), we only need to show that condition (\ref{cor:i}) implies Definition \ref{def_informativity} condition (\ref{def:i}). If condition (\ref{cor:i}) holds then by Lemma \ref{lemma:lagAndTini}, Definition \ref{def_informativity} condition (\ref{def:i}) holds, proving the claim.
\end{proof}

We now focus on developing sufficient conditions for informativity for unique prediction. Verifying $\Tini\geq \Lmax$ requires knowledge of $\Lmax$, which depends on the set of explaining systems. Computing $\Lmax$ from data is in general difficult since the set of explaining systems can be challenging to characterize. However, we only require an upper bound on $\Lmax$ to verify $\Tini\geq \Lmax$. Therefore, it would be sufficient to check if $\Tini\geq L$, where $L$ is the upper bound on the lag of the true system $\B^*$. This is in general conservative, but a tighter upper bound on $\Lmax$ can be computed from data leveraging the upper bound on the system order $N$~\cite{camlibel_beyond_2024}.

We now seek sufficient conditions for existence of a $y^\f\in\mathbb{R}^{p\Tf}$ such that $(u^\ini, u^\f, y^\ini, y^\f)\in\mathcal{B}_{[0, \Tini + \Tf - 1]}$ for all explaining systems $\mathcal{B}\in\Sigma_\data\cap\Sigma_\ini$. To do this, we first state an intermediate lemma that characterizes some of the trajectories that all explaining systems must agree on.
\begin{lemma}\label{lemma:image_of_data}
    Let $(u, y)\in\mathcal{B}^*_{[0, T - 1]}$. Then, for all $k\in[1, T]$,
    \begin{align*}
        \im
        \begin{pmatrix}
            \genHank{k}{u}{}\\
            \genHank{k}{y}{}
        \end{pmatrix}
        \subseteq\bigcap_{\mathcal{B}\in\Sigma(u,y)}\mathcal{B}_{[0, k - 1]}.
    \end{align*}
\end{lemma}
\begin{proof}
    Let $k\in[1, T]$, $\mathcal{B}\in\Sigma(u,y)$. By time-invariance of $\B$, we have that
    \begin{align*}
        \begin{pmatrix}
            u_{[t_0, t_0 + k - 1]}\\
            y_{[t_0, t_0 + k - 1]}
        \end{pmatrix}
        \in\mathcal{B}_{[t_0, t_0 + k - 1]}\subseteq\mathcal{B}_{[0, k - 1]}
    \end{align*}
    for every $t_0\in[0, T - k]$. Thus, by linearity of $\B$ and definition of Hankel matrices, 
    \begin{equation}\label{eq:imHankelInB}
        \im
        \begin{pmatrix}
            \genHank{k}{u}{}\\
            \genHank{k}{y}{}
        \end{pmatrix}
        \subseteq\mathcal{B}_{[0, k - 1]}.
    \end{equation}
    Since $\B$ is arbitrary, \eqref{eq:imHankelInB} must hold for all $\B\in\Sigma(u,y)$, proving the result.
\end{proof}
This lemma states that the image of the data matrix is contained in every explaining system's restricted behavior. Thus to verify the existence of an output response that all explaining systems agree on, we can search in the subspace defined by the image of the data matrix. Inspired by~\cite{data_driven_simulation}, we present Algorithm \ref{alg:predict}, which computes an output prediction that all explaining systems agree on.
\begin{algorithm}[htb]
\caption{Data-driven Prediction}\label{alg:predict}
\begin{algorithmic}
    \Procedure{Predict}{$u^\data, y^\data, u^\ini, y^\ini, u^\f, \ell$}
    \Input $u^\data, y^\data, u^\ini, y^\ini, u^\f, \ell\in\wholes$
    \State Construct the matrices
    \[
        \begin{pmatrix}
            \Up\\
            \Uf\\
            \Yp\\
            \Yf
        \end{pmatrix}
        =
        \begin{pmatrix}
            \genHank{\ell + \Tf}{u^\data}{} & \genHank{\ell + \Tf}{u^\ini}{}\\
            \genHank{\ell + \Tf}{y^\data}{} & \genHank{\ell + \Tf}{y^\ini}{}
        \end{pmatrix}
    \]
    where $\Up\in\R^{m\ell\times \bullet}$, $\Uf\in\R^{m\Tf\times \bullet}$, $\Yp\in\R^{p\ell\times \bullet}$, and $\Yf\in\R^{p\Tf\times \bullet}$.
    \State $\bar{u}^\ini\gets u^\ini_{[\Tini - \ell, \Tini - 1]}$
    \State $\bar{y}^\ini\gets y^\ini_{[\Tini - \ell, \Tini - 1]}$
    \If{there exists $g$ such that
    \[
    \begin{pmatrix}\bar{u}^\ini\\ u^\f\\ \bar{y}^\ini\end{pmatrix}=\begin{pmatrix}\Up \\  \Uf\\ \Yp\end{pmatrix} g
    \]}
    \State \textbf{return} $y^\f = \Yf g$
    \Else
        \State \textbf{return} null
    \EndIf
    \EndProcedure
\end{algorithmic}
\end{algorithm}

The algorithm implicitly infers $m, p, \Tini, \Tf$ from its inputs. The following lemma gives a sufficient condition for informativity for unique prediction in terms of the output of Algorithm \ref{alg:predict}.

\begin{lemma}\label{lemma:ini_and_data_in_use}
    \assumptions If $T_\ini \geq \Lmax$ and \textsc{Predict}$(u^\data, y^\data, u^\ini, y^\ini, u^\f, \Lmax)$ returns a $y^\f$, then \inform \ and $y^\f$ is the unique prediction.
\end{lemma}
\begin{proof}
    Assume \textsc{Predict}$(u^\data, y^\data, u^\ini, y^\ini, u^\f, \Lmax)$ returns a $y^\f$. To show that \inform, we verify the conditions in Lemma~\ref{lem:alt_definition}. By assumption, $\Tini\geq\Lmax$. By Lemma~\ref{lemma:image_of_data},
    \begin{equation*}
        \im
        \begin{pmatrix}
            \genHank{\Lmax + \Tf}{u^\data}{}\\
            \genHank{\Lmax + \Tf}{y^\data}{}
        \end{pmatrix}
        \subseteq\bigcap_{\mathcal{B}\in\Sigma_\data}\mathcal{B}_{[0, \Lmax + \Tf - 1]}
    \end{equation*}
    and
    \begin{equation*}
        \im
        \begin{pmatrix}
            \genHank{\Lmax + \Tf}{u^\ini}{}\\
            \genHank{\Lmax + \Tf}{y^\ini}{}
        \end{pmatrix}
        \subseteq\bigcap_{\mathcal{B}\in\Sigma_\ini}\mathcal{B}_{[0, \Lmax + \Tf - 1]}.
    \end{equation*}
    By~\cite[\S11, Theorem 11]{halmos_finite_vector_spaces}, the intersection of any collection of subspaces is also a subspace. Thus,
    \begin{equation*}
    \begin{aligned}
        &\im
        \begin{pmatrix}
            \genHank{\Lmax + \Tf}{u^\data}{} & \genHank{\Lmax + \Tf}{u^\ini}{}\\
            \genHank{\Lmax + \Tf}{y^\data}{} & \genHank{\Lmax + \Tf}{y^\ini}{}
        \end{pmatrix}\\
        &\subseteq\bigcap_{\mathcal{B}\in\Sigma_\data\cap\Sigma_\ini} \mathcal{B}_{[0, \Lmax + \Tf - 1]}.
        \end{aligned}
    \end{equation*}
    Therefore, by definition of $\Up,\Uf,\Yp,\Yf$ in Algorithm \ref{alg:predict},
    \[
        \im
        \begin{pmatrix}
           \Up\\
           \Uf\\
           \Yp\\
           \Yf
        \end{pmatrix}
        \subseteq\bigcap_{\mathcal{B}\in\Sigma_\data\cap\Sigma_\ini} \mathcal{B}_{[0, \Lmax + \Tf - 1]}.
    \]
    From the above and by definition of $\bar{u}^\ini,\bar{y}^\ini$ from \textsc{Predict}, $y^\f$ must be such that $(\bar{u}^\ini,u^\f,\bar{y}^\ini,y^\f)\in\B_{[0,\Lmax+\Tf-1]}$ for all $\B\in\Sigma_\data\cap\Sigma_\ini$. Since $(\bar{u}^\ini, \bar{y}^\ini)$ is the final $\Lmax$ terms of $(u^\ini, y^\ini)$, we use Lemma \ref{lemma:weaving} to weave them together, giving us $(u^\ini, u^\f, y^\ini. y^\f)\in\B_{[0, \Tini + \Tf - 1]}$. Thus \inform.
\end{proof}
Building a non-void, depth $\ell$ Hankel matrix requires a trajectory of length at least the depth $\ell$. In the case where $T\leq\Lmax$ and $\Tini = \Lmax$ we will only ever have informativity if $u^\ini, y^\ini, u^\f$ are all zero trajectories. Thus, if any of the initial trajectory or future inputs are non-zero, the sufficient condition in Lemma~\ref{lemma:ini_and_data_in_use} will only hold if $\max\{T, \Tini\}\geq\Lmax + \Tf$. This indicates that if $\Tf$ is large, then the amount of data or the length of initial trajectory would also need to be large. Thus we cannot apply this lemma if $\Tf>\max\{T, \Tini\} - \Lmax$. However, we may still be able to do prediction in this case. We present an example which shows it is not necessary to have $\max\{T, \Tini\}\geq\Lmax + \Tf$ for informativity.
\begin{ex}\label{ex:weaving_useful}
    Consider $\B^*$ with state-space representation \eqref{eq:sys1} where $n = m = p = 1$. Take $L = 1$ and $N = 1$. Let $T = 3$ and consider the data trajectory $(u^\data, y^\data)\in\B^*_{[0, T - 1]}$ with \eqref{eq:data}. Let $\Tini = 1$. Consider the initial trajectory $(u^\ini, y^\ini)\in\B^*_{[0, \Tini - 1]}$ with $u^\ini = -2$ and $y^\ini = 1$. Let $\Tf = 6$. Consider the future inputs $u^\f(t) = 2(-1)^t$. Here we have $\Tf > T$, which indicates that Lemma \ref{lemma:ini_and_data_in_use} cannot be used since the data matrices would be void. However, in Example \ref{ex:predict-without-ID}, we demonstrated that for $y^\f(0) = -1, y^\f(1) = 1$ we have $(u^\ini, u^\f(0), u^f(1), y^\ini, y^\f(0), y^f(1))\in\B_{[0, 2]}$ for all $\B\in\Sigma_\data\cap\Sigma_\ini$. By definition of \gls{LTI} behaviors, we must also have $(u^\ini, u^\f(0), y^\ini, y^\f(0))$, $(u^\f(0), u^\f(1), y^\f(0), y^\f(1))\in\B_{[0, 1]}$. Since $L = 1$ and $\ell(\B^*) = 1$, we must have $\Lmax = 1$. Since $(u^\ini, y^\ini) = (u^\f(1), y^\f(1))$, we can apply Lemma \ref{lemma:weaving} and repeatedly weave $(u^\ini, u^\f(0), y^\ini, y^\f(0))$ and $(u^\f(0), u^\f(1), y^\f(0), y^\f(1))$ to show that $y^\f(t) = (-1)^{t + 1}$ satisfies $(u^\ini, u^\f, y^\ini, y^\f)\in\B_{[0, \Tini + \Tf - 1]}$ for all $\B\in\Sigma_\data\cap\Sigma_\ini$. Furthermore, we can take $\Tf$ to be arbitrarily large with $u^\f(t) = 2(-1)^t$ and still uniquely predict the future output response $y^\f$. 
\end{ex}
Example~\ref{ex:weaving_useful} demonstrates that the process of weaving (Lemma~\ref{lemma:weaving}) can give informativity in situations where Lemma~\ref{lemma:ini_and_data_in_use} cannot be used. Algorithm~\ref{alg:weaving} is designed to compute predictions one step at a time and then weave them together to output the full $\Tf$-length prediction.

\begin{algorithm}[t]
\caption{Data-driven Prediction and Weaving}\label{alg:weaving}
\begin{algorithmic}
    \Procedure{PredictAndWeave}{$u^\data, y^\data, u^\ini, y^\ini, u^\f, \ell$}
    \Input $u^\data, y^\data, u^\ini, y^\ini, u^\f, \ell$
    \State $\bar{u}^0\gets u^\ini$
    \State $\bar{y}^0\gets y^\ini$
    \For{$t = 0, 1, \dots, \Tf - 1$} 
        \State $y^\f(t) = $ \textsc{Predict}$(u^\data, y^\data, \bar{u}^t, \bar{y}^t, u^\f(t),\ell)$
        \If{$y^\f(t)$ is null}
            \State \textbf{return} null
        \Else
            \State $\bar{u}^{t + 1}\gets(\bar{u}^t, u^\f(t))$
            \State $\bar{y}^{t + 1}\gets(\bar{y}^t, y^\f(t))$
        \EndIf
    \EndFor
    \State \textbf{return} $y^\f$
    \EndProcedure
\end{algorithmic}
\end{algorithm}
The following theorem provides a sufficient condition for informativity that can be applied regardless of how long the future input sequence is.

\begin{theorem}\label{main_theorem}
    \assumptions If $\Tini\geq\Lmax$ and \textsc{PredictAndWeave}$(u^\data, y^\data, u^\ini, y^\ini, u^\f, \Lmax)$ returns a $y^\f$, then \inform \ and $y^\f$ is the unique prediction.
\end{theorem}
\begin{proof}
    The requirement of $\Tini\geq\Lmax$ in the theorem statement implies condition (\ref{cor:i}) in Lemma \ref{lem:alt_definition}. We now want to show condition (\ref{cor:ii}) of Lemma \ref{lem:alt_definition}. For each $t\in[0, \Tf - 1]$, take $\bar{u}^t$ and $\bar{y}^t$ from Algorithm \ref{alg:weaving}. By applying Lemma \ref{lemma:ini_and_data_in_use} for each $t$, $(u^\data, y^\data)$ is informative for unique prediction of $u^\f(t)$ from $(\bar{u}^t, \bar{y}^t)$. By definition, for all $t\in[0, T - 1]$, we must have $(\bar{u}^t, u^\f(t), \bar{y}^t, y^\f(t))\in\mathcal{B}_{[0, \Tini + t]}$ for all $\mathcal{B}\in\Sigma_\data\cap\Sigma_\ini$. If we take $t = \Tf - 1$ then we end up with Lemma \ref{lem:alt_definition} condition (\ref{cor:ii}), indicating \inform. 
\end{proof}

\section{Numerical Example}\label{sec:numerical_example}
To demonstrate the effectiveness of Algorithm~\ref{alg:weaving} for determining informativity for unique prediction, we construct two numerical examples: one in which the sufficient condition in Theorem~\ref{main_theorem} is satisfied, and one where it is not satisfied . Let the the true system $\B^*$ have state-space representation,
\begin{subequations}
    \begin{align*}
        x(t + 1) &= 
        \begin{pmatrix}
            1 & 1\\
            -1 & -0.5
        \end{pmatrix}
        x(t) + 
        \begin{pmatrix}
            1 & 1\\
            0 & 1
        \end{pmatrix}
        u(t),\\
        y(t) &= 
        \begin{pmatrix}
            1 & 0
        \end{pmatrix}
        x(t).
    \end{align*}
\end{subequations}
Take $L = N = 2$. Let $T = 8$ and consider the data trajectory $(u^\data, y^\data)\in\B^*_{[0, T - 1]}$ shown in Fig.~\ref{fig:numerical_example_data}.
Let $\Tini = 2$. Consider the initial trajectory $(u^\ini, y^\ini)\in\B^*_{[0, \Tini -1]}$ with $u^\ini(0) = (6, 2), u^\ini(1) = (-1, 5)$ and $y^\ini(0) = y^\ini(1) = 0$. Let $\Tf = 20$. Fig.~\ref{fig:numerical_example_informative_trajectory} shows the future inputs $u^\f$ and prediction computed using Algorithm \ref{alg:weaving}. We observe that the prediction coincides with the true output response from $\B^*$. We can also again see the importance of weaving as Algorithm \ref{alg:predict} would return null in this case as all data matrices would be void.
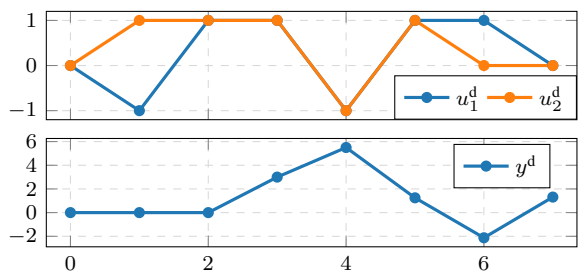
\begin{figure}[t]
    \centering    
    \begin{tikzpicture}
    \begin{groupplot}[
        group style={
            group size=1 by 2,     
            vertical sep=0.25cm,   
        },
        width=\columnwidth,                
        height=0.35\columnwidth,              
        grid=major,
        grid style={dashed, gray!30},
        enlarge x limits=0.05,
        tick label style={font=\footnotesize},
        title style={font=\bfseries}
    ]

    \nextgroupplot[
        legend style={at={(1,0)}, anchor=south east, font=\footnotesize, legend columns=-1},
        xticklabels=\empty,
        mark size=1.5pt
    ]
    
    \addplot[color=matblue, very thick, mark=*] table [x=t, y=u1, col sep=comma] {Plots/data.csv};
    \addlegendentry{$u^\data_1$}
    
    \addplot[color=matorange, very thick, mark=*] table [x=t, y=u2, col sep=comma] {Plots/data.csv};
    \addlegendentry{$u^\data_2$}

    \nextgroupplot[
        legend style={at={(0.95,0.95)}, anchor=north east, font=\footnotesize},
        mark size=1.5pt
    ]
    
    \addplot[color=matblue, very thick, mark=*] table [x=t, y=y, col sep=comma] {Plots/data.csv};
    \addlegendentry{$y^\data$}

    \end{groupplot}
\end{tikzpicture}
    \caption{The top plot shows the input data $u^\data$. The bottom plot shows the output data $y^\data$ corresponding to $u^\data$.}
    \label{fig:numerical_example_data}
\end{figure}

We examine a different sequence of $\Tf = 20$ future inputs $u^\f$ shown in Fig.~\ref{fig:numerical_example_noninformative}. We keep the setup of the problem the same with the same data and initial trajectories as above. Applying Algorithm~\ref{alg:weaving} returns a null value. Since Theorem~\ref{main_theorem} is only a sufficient condition, we in general cannot conclude whether the data is informative for unique prediction. We compute an explaining system of the data,
\begin{equation*}
    \B=\left\{ 
    \begin{pmatrix}
        u\\
        y
    \end{pmatrix}
    \;\middle|\;
{\small\begin{array}{@{}l@{}}
   y(t) = 0.648y(t - 1) - 0.324y(t - 2)\\
   +
   \begin{pmatrix}
       -3.259 \\
       4.022
   \end{pmatrix}^\top u(t)
   +
   \begin{pmatrix}
       1.987\\
       1.225
   \end{pmatrix}^\top u(t - 1)\\
   +
   \begin{pmatrix}
       1.041\\
       0.066
   \end{pmatrix}^\top u(t - 2)\qquad \forall t\in\wholes
\end{array}}
\right\},
    \end{equation*}
and apply $u^\f$ from $(u^\ini, y^\ini)$ and observe the output trajectory $y^\f$ shown in Fig.~\ref{fig:numerical_example_noninformative}. On the other hand, applying the input $u^\f$ from $(u^\ini, y^\ini)$ on the true system $\B^*$, we observe a different output trajectory $y^*$ shown in Fig.~\ref{fig:numerical_example_noninformative}. Furthermore, we can find other explaining systems that have output responses that diverge. Therefore, we can conclude that the data are not informative for unique prediction. 

This highlights the potential consequences of blindly predicting using any explaining system (as is commonly done in certainty equivalence control). What we see is that different explaining systems can have output responses that are drastically different, which could lead to unsafe behavior if the predictions were used in a data-driven control context.

\section{Conclusion}\label{sec:future_work}
We presented a notion of informativity for unique prediction and gave sufficient conditions weaker than persistency of excitation that enable unique prediction. We provided algorithms that directly compute predictions from data. A numerical example was shown that illustrates the ability to uniquely predict without data that are informative for identification. Future work focuses on classifying the set of future input sequences that can be predicted, expanding the result to noisy data, and non-\gls{LTI} systems.

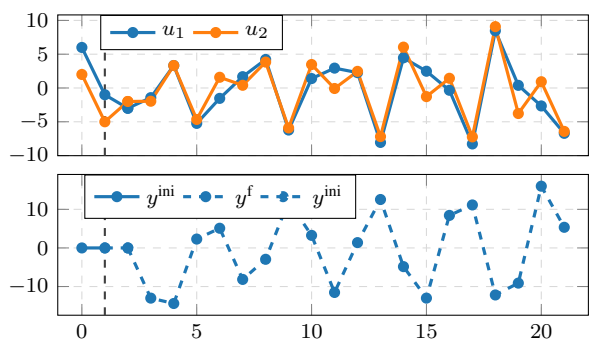
\begin{figure}[t]
    \centering
    \begin{tikzpicture}
    \begin{groupplot}[
        group style={
            group size=1 by 2,     
            vertical sep=0.25cm,  
        },
        width=\columnwidth,                
        height=0.4\columnwidth,              
        grid=major,
        grid style={dashed, gray!30},
        enlarge x limits=0.05,
        tick label style={font=\footnotesize},
        title style={font=\bfseries}
    ]

    \nextgroupplot[
        xticklabels=\empty,
        legend style={at={(0.08,1)}, anchor=north west, font=\footnotesize, legend columns=-1},
        mark size=1.5pt
    ]
    
    \addplot[color=matblue, very thick, mark=*] table [x=t, y=u1, col sep=comma] {Plots/informative_ini.csv};
    \addlegendentry{$u_1$}
    
    \addplot[color=matorange, very thick, mark=*] table [x=t, y=u2, col sep=comma] {Plots/informative_ini.csv};
    \draw[dashed, darkgray, thick] (1, \pgfkeysvalueof{/pgfplots/ymin}) -- (1, \pgfkeysvalueof{/pgfplots/ymax});
    \addlegendentry{$u_2$}

    \addplot[color=matblue, very thick, mark=*] table [x=t, y=u1, col sep=comma] {Plots/informative_future.csv};

    \addplot[color=matorange, very thick, mark=*] table [x=t, y=u2, col sep=comma] {Plots/informative_future.csv};

    \nextgroupplot[
        legend style={at={(0.05,1)}, anchor=north west, font=\footnotesize, legend columns=-1},
        mark size=1.5pt
    ]

    \addplot[color=matblue, very thick, mark=*, restrict x to domain=0:1] table [x=t, y=y, col sep=comma] {Plots/informative_ini.csv};
    \addlegendentry{$y^\ini$}
    \draw[dashed, darkgray, thick] (1, \pgfkeysvalueof{/pgfplots/ymin}) -- (1, \pgfkeysvalueof{/pgfplots/ymax});
    
    \addplot[color=matblue, very thick, dashed, mark=*, mark options={solid}] table [x=t, y=y, col sep=comma] {Plots/informative_future.csv};
    \pgfkeysvalueof{/pgfplots/ymax});
    \addlegendentry{$y^\f$}

    \addplot[color=matblue, very thick, dashed, mark=*, restrict x to domain=1:2] table [x=t, y=y, col sep=comma] {Plots/informative_ini.csv};
    \addlegendentry{$y^\ini$}

    \end{groupplot}
\end{tikzpicture}
    \caption{A future prediction is provided by applying Algorithm \ref{alg:weaving} that is accurate for all explaining systems. The top plot shows the initial input sequence $u^\ini$ followed by the future input sequence $u^\f$. The bottom plot shows the initial output sequence $y^\ini$ followed by the predicted future output $y^\f$ in response to $u^\f$.}
    \label{fig:numerical_example_informative_trajectory}
\end{figure}
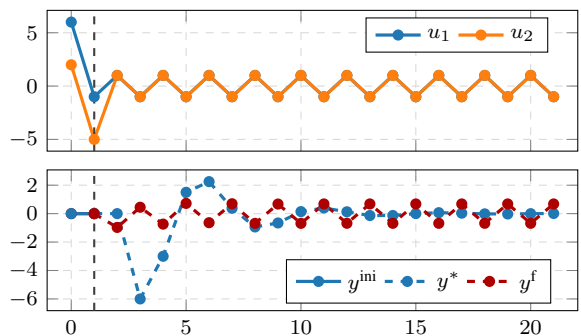
\begin{figure}[t]
    \centering
    \begin{tikzpicture}
    \begin{groupplot}[
        group style={
            group size=1 by 2,     
            vertical sep=0.25cm,   
        },
        width=\columnwidth,                
        height=0.4\columnwidth,              
        grid=major,
        grid style={dashed, gray!30},
        enlarge x limits=0.05,
        tick label style={font=\footnotesize},
        title style={font=\bfseries}
    ]

    \nextgroupplot[
        legend style={at={(0.6,0.95)}, anchor=north west, font=\footnotesize, legend columns=-1},
        xticklabels=\empty,
        mark size=1.5pt
    ]
    
    \addplot[color=matblue, very thick, mark=*] table [x=t, y=u1, col sep=comma] {Plots/noninformative.csv};
    \addlegendentry{$u_1$}
    
    \addplot[color=matorange, very thick, mark=*] table [x=t, y=u2, col sep=comma] {Plots/noninformative.csv};
    \draw[dashed, darkgray, thick] (1, \pgfkeysvalueof{/pgfplots/ymin}) -- (1, \pgfkeysvalueof{/pgfplots/ymax});
    \addlegendentry{$u_2$}

    \nextgroupplot[
        legend style={at={(0.95,0.05)}, anchor=south east, font=\footnotesize, legend columns=-1},
        mark size=1.5pt
    ]

    \addplot[color=ybadcolor, very thick, mark=*, restrict x to domain=0:1, solid, mark options={solid}, forget plot] table [x=t, y=ybad, col sep=comma] {Plots/noninformative.csv};
    
    \addplot[color=matblue, very thick, mark=*, restrict x to domain=0:1] table [x=t, y=ygood, col sep=comma] {Plots/noninformative.csv};
    \pgfkeysvalueof{/pgfplots/ymax});
    \addlegendentry{$y^\ini$}

    \addplot[color=matblue, very thick, mark=*, restrict x to domain=1:30, dashed, mark options={solid}] table [x=t, y=ygood, col sep=comma] {Plots/noninformative.csv}; \pgfkeysvalueof{/pgfplots/ymax});
    \addlegendentry{$y^*$}

    \addplot[color=ybadcolor, very thick, mark=*, restrict x to domain=1:30, dashed, mark options={solid}] table [x=t, y=ybad, col sep=comma] {Plots/noninformative.csv};
    \draw[dashed, darkgray, thick] (1, \pgfkeysvalueof{/pgfplots/ymin}) -- (1, \pgfkeysvalueof{/pgfplots/ymax});
    \addlegendentry{$y^\f$}

    \end{groupplot}
\end{tikzpicture}
    \caption{The top plot shows the initial input sequence $u^\ini$ followed by a future input sequence $u^\f$ for which Algorithm \ref{alg:weaving} returns null. The bottom plot shows the initial output sequence $y^\ini$ followed by the true system $\B^*$'s response $y^*$ and the response $y^\f$ predicted using the explaining system $\B$.}
    \label{fig:numerical_example_noninformative}
\end{figure}

\section*{Acknowledgment}
ChatGPT, Gemini, and Claude were used to aid in \LaTeX~formatting, improving grammar, and debugging the code used in the numerical example. The authors edited the content and take responsibility for the contents of the paper.

\bibliographystyle{IEEEtran}

\begin{thebibliography}{10}
\providecommand{\url}[1]{#1}
\csname url@rmstyle\endcsname
\providecommand{\newblock}{\relax}
\providecommand{\bibinfo}[2]{#2}
\providecommand\BIBentrySTDinterwordspacing{\spaceskip=0pt\relax}
\providecommand\BIBentryALTinterwordstretchfactor{4}
\providecommand\BIBentryALTinterwordspacing{\spaceskip=\fontdimen2\font plus
\BIBentryALTinterwordstretchfactor\fontdimen3\font minus \fontdimen4\font\relax}
\providecommand\BIBforeignlanguage[2]{{%
\expandafter\ifx\csname l@#1\endcsname\relax
\typeout{** WARNING: IEEEtran.bst: No hyphenation pattern has been}%
\typeout{** loaded for the language `#1'. Using the pattern for}%
\typeout{** the default language instead.}%
\else
\language=\csname l@#1\endcsname
\fi
#2}}

\bibitem{hou2013datadriven}
Z.-S. Hou and Z.~Wang, ``From model-based control to data-driven control: Survey, classification and perspective,'' \emph{Information Sciences}, vol. 235, pp. 3--35, 2013.

\bibitem{ljung1998system}
L.~Ljung, ``System identification,'' in \emph{Signal analysis and prediction}.\hskip 1em plus 0.5em minus 0.4em\relax Springer, 1998, pp. 163--173.

\bibitem{vapnik2013nature}
V.~Vapnik, \emph{The nature of statistical learning theory}.\hskip 1em plus 0.5em minus 0.4em\relax Springer science \& business media, 2013.

\bibitem{markovsky2021behavioral}
I.~Markovsky and F.~D{\"o}rfler, ``Behavioral systems theory in data-driven analysis, signal processing, and control,'' \emph{Annual Reviews in Control}, vol.~52, pp. 42--64, 2021.

\bibitem{data_driven_simulation}
\BIBentryALTinterwordspacing
I.~Markovsky and P.~Rapisarda, ``Data-driven simulation and control,'' \emph{International Journal of Control}, vol.~81, no.~12, pp. 1946--1959, 2008. [Online]. Available: \url{https://doi.org/10.1080/00207170801942170}
\BIBentrySTDinterwordspacing

\bibitem{DeePC}
J.~Coulson, J.~Lygeros, and F.~Dörfler, ``Data-enabled predictive control: In the shallows of the {DeePC},'' in \emph{2019 18th European Control Conference (ECC)}, 2019, pp. 307--312.

\bibitem{favoreel1999spc}
W.~Favoreel, B.~De~Moor, and M.~Gevers, ``{SPC}: Subspace predictive control,'' \emph{IFAC Proceedings Volumes}, vol.~32, no.~2, pp. 4004--4009, 1999.

\bibitem{poldermanBook}
\BIBentryALTinterwordspacing
J.~W. Polderman and J.~C. Willems, \emph{\BIBforeignlanguage{en}{Introduction to {Mathematical} {Systems} {Theory}}}, ser. Texts in {Applied} {Mathematics}, J.~E. Marsden, L.~Sirovich, M.~Golubitsky, W.~Jäger, and F.~John, Eds.\hskip 1em plus 0.5em minus 0.4em\relax New York, NY: Springer New York, 1998, vol.~26. [Online]. Available: \url{http://link.springer.com/10.1007/978-1-4757-2953-5}
\BIBentrySTDinterwordspacing

\bibitem{willems2005note}
J.~C. Willems, P.~Rapisarda, I.~Markovsky, and B.~L. De~Moor, ``A note on persistency of excitation,'' \emph{Systems \& Control Letters}, vol.~54, no.~4, pp. 325--329, 2005.

\bibitem{informativity}
H.~J. Van~Waarde, J.~Eising, M.~K. Camlibel, and H.~L. Trentelman, ``The informativity approach: To data-driven analysis and control,'' \emph{IEEE Control Systems Magazine}, vol.~43, no.~6, pp. 32--66, 2023.

\bibitem{weaving}
\BIBentryALTinterwordspacing
I.~Markovsky, J.~C. Willems, P.~Rapisarda, and B.~L. De~Moor, ``\BIBforeignlanguage{en}{Algorithms for deterministic balanced subspace identification},'' \emph{\BIBforeignlanguage{en}{Automatica}}, vol.~41, no.~5, pp. 755--766, May 2005. [Online]. Available: \url{https://linkinghub.elsevier.com/retrieve/pii/S0005109804003036}
\BIBentrySTDinterwordspacing

\bibitem{camlibel_beyond_2024}
\BIBentryALTinterwordspacing
K.~Camlibel and P.~Rapisarda, ``Beyond the fundamental lemma: from finite time series to linear system,'' May 2024, arXiv:2405.18962 [math]. [Online]. Available: \url{http://arxiv.org/abs/2405.18962}
\BIBentrySTDinterwordspacing

\bibitem{halmos_finite_vector_spaces}
P.~R. Halmos, \emph{Finite-Dimensional Vector Spaces}, ser. Undergraduate Texts in Mathematics.\hskip 1em plus 0.5em minus 0.4em\relax Springer, 1993.

\end{thebibliography}

\end{document}